\documentclass[12pt]{amsart}

\usepackage{latexsym,amssymb,amsmath,graphics}
\usepackage[dvips]{graphicx}

\def\cc{{\mathcal C}}

\def\hh{{\mathcal H}}

\def\mm{{\mathcal M}}
\def\nn{{\mathcal N}}
\def\pp{{\mathcal P}}

\def\ffi{\varphi}
\def\eps{\varepsilon}
\def\dst{\displaystyle}

\renewcommand{\Im}{\mathrm{Im}\,}

\DeclareMathOperator{\vect}{span}
\def\C{{\mathbb{C}}}

\def\N{{\mathbb{N}}}

\def\Q{{\mathbb{Q}}}
\def\R{{\mathbb{R}}}
\def\S{{\mathbb{S}}}

\def\Z{{\mathbb{Z}}}

\def\e{{\mathbf{e}}}

\newcommand{\abs}[1]{{\left|{#1}\right|}}
\newcommand{\scal}[1]{{\left\langle{#1}\right\rangle}}

\newtheorem{lemma}{Lemma}[section]
\newtheorem{proposition}[lemma]{Proposition}
\newtheorem{theorem}[lemma]{Theorem}

\theoremstyle{definition}
\newtheorem{definition}[lemma]{Definition}
\newtheorem{example}[lemma]{Example}

\theoremstyle{remark}
\newtheorem{remark}[lemma]{Remark}

\begin{document}

\title[Phase retrieval for Herglotz functions]{The phase retrieval problem for solutions of the Helmholtz equation}

\author{Philippe Jaming \& Salvador P\'erez-Esteva}

\address{P.J. \,: Univ. Bordeaux, IMB, UMR 5251, F-33400 Talence, France.
CNRS, IMB, UMR 5251, F-33400 Talence, France.}
\email{Philippe.Jaming@math.u-bordeaux.fr}

\address{S.P.E.\,: Instituto de Matem\'{a}ticas, Unidad Cuernavaca\\
Universidad Nacional Aut\'{o}noma de M\'{e}xico\\
Cuernavaca\\ Morelos 62251\\ M\'{e}xico}
\email{spesteva@im.unam.mx}

\begin{abstract}
In this paper we consider the phase retrieval problem for Herglotz functions, that is, solutions of the Helmholtz equation
$\Delta u+\lambda^2u=0$ on domains $\Omega\subset\R^d$, $d\geq2$. In dimension $d=2$, if $u,v$ are two such solutions 
then $|u|=|v|$ implies that either $u=cv$ or $u=c\bar v$ for some $c\in\C$ with $|c|=1$. In dimension $d\geq3$,
the same conclusion holds under some restriction on $u$ and $v$:
either they are real valued or zonal functions or have non vanishing mean.
\end{abstract}



\maketitle


\section{Introduction}
The phase retrieval problem consists in reconstructing a function from its modulus or the modulus of some
transform of it (frame coefficient, Fourier transform,...) and some structural information on the function ({\it e.g.}
to be compactly supported). This kind of problems occur in many scientific fields
such as microscopy, holography, crystallography, neutron radiography,
optical coherence tomography, optical design, radar signal processing ,
quantum mechanics to name a few. 
We refer to the books \cite{Hu,St}, the review articles \cite{KST,Mi,Fi,LBL} and to the introduction
of our previous paper \cite{Ja} for descriptions of various instances of this problems, some solutions to it 
(both theoretical and numerical) and for further references.

The problem can be split into two main questions:

-- design algorithms that allow to reconstruct at least one solution. 

-- obtain uniqueness results (up to obvious invariants of the problem like the multiplication by a constant phase factor).

After having long been ignored by mathematicians, recent progress on the algorithmic aspect of the problem \cite{CSV,WdAM}
has triggered a lot of attention to this problem. While the design of numerical algorithms allowing to reconstruct one
solution is of course essential, the task can only be complete once one is certain to reconstruct {\em all}
solutions of interest. This is generally not possible as long as uniqueness is not guarantied and plainly justifies
the second part of the problem. Uniqueness is also usefull in order to have stability results in presence of additive noise. In this paper, we will only deal with the uniqueness aspect of the problem.
More precisely, the phase retrieval problem is extremely common in optical sciences, among other reasons, this is due to the lack of sensitivity
of optical measurement instruments to phase. It turns out that optical signals are solution of partial differential equations
and our aim is to show that this information can be of some use in the phase retrieval problem.
 In our previous work \cite{Ja},
we considered solutions of the free Shr\"odinger equation. Our aim here is to pursue a similar study for
solutions of the Helmholtz equation $\Delta u+\lambda^2u=0$ in a domain $\Omega\subset\R^d$, $d\geq 2$ where by {\em domain} we mean an open connected subset of $\R^d$. Recall that this
equation is obtained by reducing the wave equation to monochromatic waves and that $|u|^2$ when $u$ is a solution
of this equation is the intensity of the monochromatic wave and does not vary with time.
Note also that a different version of the phase retrieval problem for solutions of the Helmholtz equation
has been recently studied in \cite{KS,Kl1,Kl2,KR2} and for the Shr\"odinger equation in \cite{KR} and references therein.

Finally, note that, up to a renormalization, we may restrict attention to the case $\lambda=1$. Up to a translation, we may also assume that $0\in\Omega$. We are thus concerned with the following problem:

\medskip

\noindent{\bf Phase Retrieval Problem for the Helmholtz Equation.}
{\sl Let $d\geq2$ and $0\in\Omega\subset\R^d$ be a domain. Let $u,v$ be two solutions of the Helmholtz equation on $\Omega$ 
$$
\Delta u+u=0.\leqno(H)
$$
Does $|v|=|u|$ imply that $v=cu$ or $v=c\bar u$ for some $c\in\C$ with $|c|=1$.}

\medskip

Of course, if $v=cu$ or $v=c\bar u$ then $v$ is also a solution of $(H)$ and $|v|=|u|$. We will say that
$v$ is a trivial solution of the phase retrieval solution for $u$. This problem, as stated still eludes us. Our aim here is to
show that, in many instances, the problem has only trivial solutions.

Our main result is then the following:

\medskip

\noindent{\bf Main Theorem.} {\sl 
Let $d\geq2$ and $0\in\Omega\subset\R^d$ be a domain. Let $u,v$ be two solutions of the Helmholtz equation on $\Omega$ 
$$
\Delta u+u=0\quad,\quad\Delta v+v=0
$$
on $\Omega$.
Assume one of the following holds

--- $u$ and $v$ are real valued;

--- $u$ has non-zero mean;

--- The dimension is $d=2$;

--- $d\geq 3$ and $u,v$ are zonal functions.

Then there exists a constant $c\in\C$ such that
either $v=cu$ on $\Omega$ or $v=c\bar u$ on $\Omega$ {\it i.e.} $v$ is a trivial solution of the phase retrieval problem.}

\medskip

Recall that a zonal function is a function of the form $\ffi(\scal{x,x_0})$. Such functions are sometimes also called ridge functions.

\medskip

The remaining of this paper is organized as follows: in the next section, we gather all information we need about spherical harmonics and
Bessel function and then reformulate the problem in terms of spherical harmonic coefficients. The remaining of the paper is then devoted to the proofs of the various statements of the main theorem,
Section \ref{sec:d2} is devoted to the $2$-dimensional cases and Section \ref{sec:d3} to the other three cases.

\section{Preliminaries}

\subsection{Notations}

Throughout this paper, $d$ is an integer, $d\geq 2$.
The Euclidean norm in $\R^d$ is denoted by $\abs{x}=\bigl(x_1^2+\cdots+x_d^2\bigr)^{1/2}$ and we set $\S^{d-1}=\{x\in\R^d\,:\ \abs{x}=1\}$.
We denote the standard basis of $\R^d$ by $\e_1,\ldots,\e_d$.

We denote by $\N_0$ the set of non-negative integers, $\N=\N_0\cup\{0\}$. 
For $\alpha\in\N_0^d$ we use the standard multi-index notation,
$\abs{\alpha}=\alpha_1+\cdots+\alpha_d$ and
$\partial^\alpha=\partial_{x_1}^{\alpha_1}\cdots\partial_{x_d}^{\alpha_d}$.
The Laplace operator is defined by $\Delta=\dst\sum_{j=1}^d\partial_{x_j}^2$.
%

\subsection{Spherical harmonics}

Let us here gather some information on spherical harmonics as can be found in many books in harmonic analysis.
We will here take the notations from \cite[Chapter 1]{FD}.

We denote by $\pp_m^d$ the space of homogeneous polynomials of degree $m$ in $\R^d$ and
$$
\hh_m^d=\{P\in\pp_n^d\,:\ \Delta P=0\}
$$
the space of harmonic homogeneous polynomials of degree $n$ in $\R^d$. Recall that this space has dimension
$N(m)=\dst{m+d-1\choose m}-{m+d-3\choose m-2}$
with the standard convention that the second term vanishes for $m=0$ and $
m=1$. In particular $N(0)=1$, and for $m\geq 1$, $N(m)=2$ when $d=2$ while $N(m)=2m+1$ when $d=3$

We will not distinguish homogeneous polynomials and their restriction to $\S^{d-1}$.
In particular, $\hh_m^d$ and $\hh_n^d$ are orthogonal subspaces of $L^2(\S^{d-1})$ when $n\not=m$.

Recall the following definition:

\begin{definition}
	A function $f$ on $\R^d$ or $\S^{d-1}$ is said to be zonal with respect to some $\zeta_0\in\S^{d-1}$ if there exists a function $f_0$ on $\R$ such that
	$f(x)=f_0(\scal{x,\zeta_0})$.
\end{definition}

Such functions can be described in terms of a so-called {\em zonal basis}:
$$
Y_m^j(\theta)=C_m^{d/2-1}(\scal{\theta,\zeta_m^j})
$$
where $\{\zeta_m^j,m\geq 0,j=1,\ldots,N(m)\}\subset\S^{d-1}$, $\zeta_m^1=\zeta_0$ for every $m$ and 
$C_m^\lambda$ is the Gegenbauer polynomial of degree $m$ and parameter $\lambda$ are given by
$$
C_m^\lambda(z)=\sum_{k=0}^{\lfloor m/2\rfloor} (-1)^k\frac{\Gamma(m-k+\lambda)}{\Gamma(\lambda)k!(m-2k)!}(2z)^{m-2k}.
$$
For the existence of such a basis, see {\it e.g.} \cite[Theorem 3.3]{FD}. A zonal function
has then an expension in $L^2(\S^{d-1})$ in terms of $(Y_m^1)_{m\geq 0}$ only: $f$ is zonal
with respect to $\zeta_0$ if and only if $f=\sum_{m\geq 0}a_m(f)Y_m^1$.
A zonal basis need not be orthogonal however, the orthogonality property of the $\hh_m^d$'s show that there is no convergence issue in $L^2(\S^{d-1})$. Those basis have the very desirable property for our problem to be real valued.
We also recall that the zonal basis is extended to $\R^d$  by homogeneity, 
$Y_m^j(x)=\abs{x}^mC_m^{d/2-1}\left(\scal{\frac{x}{\abs{x}},\zeta_m^j}\right)$.

We will need the following simple Lemma:

\begin{lemma}
	\label{lem:zonlinind}
	Let $m,n,d$ be integers, $m,n\geq 1$, $d\geq3$ $\lambda=d/2-1$, $a,b\in\C\setminus\{0\}$, $\zeta_1,\zeta_2\in\S^{d-1}$.
	Assume that $aC_m^\lambda(\scal{\theta,\zeta_1})^2=bC_n^\lambda(\scal{\theta,\zeta_2})^2$
	then $m=n$, $\zeta_2=\pm\zeta_1$ and $a= b$.
\end{lemma}

\begin{proof}
Let $\zeta_3,\zeta_4\in\S^{d-1}$ be such that 
\begin{enumerate}
\renewcommand{\theenumi}{\roman{enumi}}
	\item $\zeta_1,\zeta_3,\zeta_4$ are form an orthonormal basis of its span;
	\item $\zeta_2\in\vect\{\zeta_1,\zeta_3\}$.
\end{enumerate}

Write $\zeta_2=\cos\ffi\zeta_1+\sin\ffi\zeta_3$ with $\ffi\in[0,2\pi)$,
so that we want to show that $\ffi=0$ or $\ffi=\pi$.

Let $\theta=\cos s\bigl(\cos t\zeta_1+\sin t\zeta_3\bigr)+\sin s\zeta_4\in\S^{d-1}$ then
$\scal{\theta,\zeta_1}=\cos t\cos s$ while $\scal{\theta,\zeta_2}=\cos(t-\ffi)\cos s$
and $aC_m^\lambda(\scal{\theta,\zeta_1})^2=bC_n^\lambda(\scal{\theta,\zeta_2})^2$ reduces to
\begin{equation}
\label{eq:geg}
aC_m^\lambda\bigl(\cos t\cos s\bigr)^2=bC_n^\lambda\bigl(\cos(t-\ffi)\cos s\bigr)^2.	
\end{equation}

As powers of the function $\cos x$ are linearly independent, it remains to look
at the highest order term in $\cos s$ to see that \eqref{eq:geg} implies
\begin{multline*}
a\left(\frac{2^m\Gamma(m+\lambda)}{m!}\right)^2\cos^{2m} t
=b\left(\frac{2^n\Gamma(n+\lambda)}{n!}\right)^2\cos^{2n} (t-\ffi).
\end{multline*}

Using again the linear independence of powers of $\cos$, this implies that
$m=n$ and thus reduces to
$$
a\cos^{2m} t=b\cos^{2m} (t-\ffi).
$$
As
$$
2^{2m}\cos^{2m} t=(e^{it}+e^{-it})^{2m}=
\sum_{j=0}^{2m}{2m\choose j}e^{i(2m-2j)t}
$$
we get $a=be^{-i(2m-2j)\ffi}$ for $j=0,\ldots,2m$.
Therefore $a=b$ and $\ffi=0$ or $\ffi=\pi$ that is $\zeta_2=\pm\zeta_1$. 
\end{proof}

\begin{remark}
Note that if $m=0$ or if $a=0$, we get $a=b$ but, of course, we can not conclude that $\zeta_2=\pm\zeta_1$.
\end{remark}

From e.g. \cite[Theorem 1.9]{FD}, we can construct another real valued basis for $d\geq 3$.
If $\alpha\in\N_0^d$ 
 is obtained as follows:
$\alpha\in\N_0^d$, $\abs{\alpha}=m$
\begin{equation}
	\label{eq:specspherharm}
p_\alpha(x)=\frac{(-1)^m}{2^m\begin{pmatrix} \frac{d-2}{2}\end{pmatrix}_m}\abs{x}^{d-2+2m}\partial^\alpha \abs{x}^{-d+2}
=x^\alpha+\abs{x}^2q_\alpha(x)
\end{equation}
with $q_\alpha$ an homogeneous polynomial of degree $m-2$ (when $m=0$ or 1, $q_\alpha=0$). Then
$$
\{p_\alpha\,:\ \alpha\in\N_0^d,\ \abs{\alpha}=m,\ \alpha_d=0\mbox{ or }1\}
$$
is a basis of $\hh_m^d$, that is not orthonormal.

For $j=0$ or $j=1$, let us write 
$\nn_m^j=\{\alpha\in\N_0^d,\ \abs{\alpha}=m,\ \alpha_d=j\}$,
and write $\nn_m=\nn_m^0\cup\nn_m^1$.

\begin{lemma}
\label{lem:specbaslinind}
For every integer $m\geq 0$, the set of polynomials
$$
\{p_\alpha^2\,:\ \alpha\in\N_0^d,\ \abs{\alpha}=m,\ \alpha_d=0\mbox{ or }1\}
$$
is linearly independent.
\end{lemma}

\begin{proof}
Note that $p_\alpha^2=x^{2\alpha}+2x^\alpha \abs{x}^2q_\alpha(x)+\abs{x}^4q_\alpha(x)^2$. Assume that 
\begin{equation}
\label{eq:linindpalph}
\sum_{\alpha\in\nn_m}\lambda_\alpha p_\alpha^2=
\sum_{\alpha\in\nn_m}\lambda_\alpha x^{2\alpha}
+2\abs{x}^2\sum_{\alpha\in\nn_m}\lambda_\alpha x^\alpha q_\alpha
+\abs{x}^4\sum_{\alpha\in\nn_m}\lambda_\alpha q_\alpha^2
=0.
\end{equation}
In this sum,  the terms $\dst\sum_{\alpha\in\nn_m^0}\lambda_\alpha x^{2\alpha}$ 
are clearly linearly independent from from all other terms (since $x^{2\e_d}$ factors in them),
thus $\dst\sum_{\alpha\in\nn_m^0}\lambda_\alpha x^{2\alpha}=0$
and finally $\lambda_\alpha=0$ if $\alpha_d\in\nn_m^0$. But then \eqref{eq:linindpalph} reduces to
$$
x^{2\e_d}\sum_{\alpha\in\nn_m^1}\lambda_\alpha x^{2\alpha-2\e_d}
+2\abs{x}^2x^{\e_d}\sum_{\alpha\in\nn_m^1}\lambda_\alpha x^{\alpha-\e_d} q_\alpha
+\abs{x}^4\sum_{\alpha\in\nn_m^1}\lambda_\alpha q_\alpha^2=0.
$$
Again, the first term is linearly independent from the two others, thus
$\dst\sum_{\alpha\in\nn_m^1}\lambda_\alpha x^{2\alpha-2\e_d}
=0$ and finally $\lambda_\alpha=0$ for $\alpha\in\nn_m^1$ as well.
\end{proof}

%

\subsection{Bessel functions}

Here we gather some basic facts about Bessel functions (of integral order) that can be found {\it e.g.} in \cite{AS}
and \cite{Wa}.

Let $n\geq 0$ be an integer and $\alpha>0$. Recall that the Bessel function $J_{n+a}$ can be defined via the power series
$$
J_{n+\alpha}(r)=\left(\frac{r}{2}\right)^{n+\alpha}\sum_{k=0}^\infty \frac{(-1)^k}{k!\Gamma(n+\alpha+k)}\left(\frac{r}{2}\right)^{2k}.
$$
Alternatively
$$
J_0(r)=\frac{1}{\pi}\int_0^\pi\cos(r\sin \theta)\,\mbox{d}\theta
$$
and, for $n\geq 1$,
$$
J_n(r)=\frac{ 2r^n}{\sqrt{\pi}2^n\Gamma(n+1/2)}\int_0^1 (1-t^2)^{n-1/2}\cos (rt)\,\mbox{d}t.
$$
Finally, recall that for a negative integer $n<0$, we may define $J_{n}=(-1)^nJ_{-n}$.

From these expressions, one immediately deduces that $|J_0(r)|\leq 1$ and, for $n\in\Z\setminus\{0\}$
$$
|J_{n+\alpha}(r)|\leq\frac{ 2}{\sqrt{\pi}\Gamma(n+\alpha+1/2)}\left(\frac{r}{2}\right)^{n+\alpha}.
$$
In particular, if $(a_n)\in\ell^2(\Z)$, the series
$$
\sum_{n\in\Z} a_nJ_n(r)
$$
is uniformly convergent over every bounded interval. Moreover, all formal computations that we will use are directly justified with these estimaes.

Next, if $n,m\in\N_0$, then ({\it see} \cite[Formula 9.1.14]{AS}, \cite[p 147]{Wa})
\begin{multline}
\label{eq:prodbessel}
J_{n+\alpha}(r)J_{m+\alpha}(r)
=\left(\frac{r}{2}\right)^{n+m+2\alpha} \\
\sum_{k=0}^\infty\frac{(-1)^k\Gamma(n+m+2a+2k+1)}{\Gamma(n+a+k+1)\Gamma(m+a+k+1)\Gamma(n+m+2a+k+1)}\left(\frac{r}{2}\right)^{2k}.
\end{multline}
As a consequence, $J_n(r)J_m(r)\sim r^{n+m+2\alpha}$ when $r\to 0$. 
Moreover, using the pointwise bound of $J_n$, we see that if $(c_{n,m})_{n,m\geq0}$ is of moderate growth then
$$
r^{-2\alpha}\sum_{n,m\geq 0}c_{n,m}J_{n+\alpha}(r)J_{m+\alpha}(r)
$$
is holomorphic in a neighborhood of $0$. Moreover if, for some $\eta>0$ and for $0\leq r<\eta$,
$$
\sum_{n,m\geq0} c_{n,m}J_{n+\alpha}(r)J_{m+\alpha}(r)=0
$$
then, for every $j\geq0$, $\sum_{n+m=j}c_{n,m}J_{n+\alpha}(r)J_{m+\alpha}(r)=0$.
But then, using \cite[p150]{Wa}:
$$
J_{n+\alpha}(r)J_{m+\alpha}(r)=\frac{2}{\pi}\int_0^{\pi/2}J_{n+m+2\alpha}(2r\cos\theta)\cos(n-m)\theta\,\mbox{d}\theta
$$
we have
\begin{eqnarray*}
0&=&\int_0^{\pi/2}J_{j+2\alpha}(2r\cos\theta)\left(\sum_{n=0}^jc_{n,j-n}\cos(2n-j)\theta\right)\,\mbox{d}\theta\\
&=&r^{j+2\alpha}\sum_{k=0}^\infty \frac{(-1)^k}{k!\Gamma(j+2\alpha+k)!}r^{2k}\\
&&\qquad\qquad\times
\int_0^{\pi/2}\left(\sum_{n=0}^jc_{n,j-n}\cos(2n-j)\theta\right)\cos^{j+2k}\theta\,\mbox{d}\theta.
\end{eqnarray*}
It follows that, for every $j,k\geq 0$,
$$
\sum_{n=0}^jc_{n,j-n}\int_0^{\pi/2}\cos(2n-j)\theta\cos^{j+2k}\theta\,\mbox{d}\theta=0
$$
that we rewrite
\begin{equation}
\label{eq:base}
\int_0^{\pi/2}\left(\cos^j\theta\sum_{0\leq n\leq j/2}\frac{c_{n,j-n}+c_{j-n,n}}{2}\cos(2n-j)\theta\right)(\cos^2\theta)^k\,\mbox{d}\theta=0.
\end{equation}
We will now use the following simple lemma:

\begin{lemma}\label{lem:dens}
Let $f\,:[a,b]\to[0,1]$ be $\cc^1$, strictly monotonic and onto, then $\{f^k\}_{k\in\N}$ is dense
in $L^2[a,b]$.
\end{lemma}

\begin {proof}
According to Hahn-Banach's, it is enough to show that, if $\ffi\in L^2[a,b]$ is such that
\begin{equation}
\label{eq:HB}
\int_a^b\ffi(s)f^k(s)\,\mbox{d}s=0
\end{equation}
for every $k\geq0$ then $\ffi=0$. But
$$
\int_a^b\ffi(s)f^k(s)\,\mbox{d}s=\int_0^1\frac{\ffi\bigl(f^{-1}(t)\bigr)}{f'\bigl(f^{-1}(t)\bigr)}t^k\,\mbox{d}t
$$
thus, \eqref{eq:HB} implies that $\dst\int_0^1\psi(t)t^k\,\mbox{d}t=0$ for every $k\geq0$ where $\dst
\psi(t)=\frac{\ffi\bigl(f^{-1}(t)\bigr)}{f'\bigl(f^{-1}(t)\bigr)}$. As $\{t^k\}_{k\in\N}$ in $L^2[0,1]$, 
we get $\psi=0$ thus $\ffi=0$.
\end{proof}

Applying Lemma \ref{lem:dens} to \eqref{eq:base} implies that, for every $j\geq 0$ and every $\theta\in(0,\pi/2)$
$$
\cos^j\theta\sum_{0\leq n\leq j/2}\frac{c_{n,j-n}+c_{j-n,n}}{2}\cos(2n-j)\theta=0.
$$
This implies that for every $j\geq 0$ and every $0\leq n\leq j$,
$c_{n,j-n}+c_{j-n,n}=0$. We have thus proved the following lemma:

\begin{lemma}
\label{lem:linind} Let $\alpha,\eta>0$.
Let $(c_{n,m})_{n,m\geq0}$ be a sequence with at most polynomial growth. The following are equivalent:
\begin{enumerate}
\renewcommand{\theenumi}{\roman{enumi}}
\item for every $0\leq r <\eta$,
$$
\sum_{n,m\geq0} c_{n,m}J_{n+\alpha}(r)J_{m+\alpha}(r)=0.
$$
\item for every $j\geq0$ and every $0\leq k\leq j$, $c_{k,j-k}+c_{j-k,k}=0$.
\end{enumerate}
\end{lemma}

We will apply this lemma in the form
$$
\sum_{n,m\geq0} c_{n,m}J_{n+\alpha}(r)J_{m+\alpha}(r)=\sum_{n,m\geq0}\tilde c_{n,m}J_{n+\alpha}(r)J_{m+\alpha}(r)
$$
if and only if, for every $j\geq0$ and every $0\leq k\leq j$, $c_{k,j-k}+c_{j-k,k}=\tilde c_{k,j-k}+\tilde c_{j-k,k}$.

%

\subsection{Reduction of the problem}

Let us first make the following observation: solutions of the Helmholtz equation are real analytic in $\Omega$, a connected set. Therefore, if
$u=cv$ (resp. $u=c\bar v$) in a ball $B(0,\eps)\subset\Omega$, then
$u=cv$ (resp. $u=c\bar v$) in the whole of $\Omega$. We may thus assume that $\Omega$ is a ball centered at $0$.

We will now need to describe the solutions of 
$$
\Delta u+u=0\leqno(H)
$$
in polar coordinates. As is well known and easily shown,
in a neighborhood $B(0,\eps)\subset\Omega$ of $0$, 
a solution $u$ of $(H)$ can be expanded as a series
\begin{equation}
\label{eq:helm1}
u(r\theta) \sim (2\pi)^{1/2}r^{-(d-2)/2}\sum_{m=0}^\infty\sum_{j=1}^{N(m)}
a_{m,j}(u) J_{\nu(m)}(r) Y_m^j(\theta)
\end{equation}
where 
$\nu(m) = m + (d - 2)/2$ and $\{Y^j_m\}_{j=1,\ldots,N(m)}$ is a basis for the
spherical harmonics of degree $m$ in $\R^d$.

Throughout the remaining of this section, $u,v$ will be two solutions
of $(H)$ in $\Omega$ such that $\abs{u}=\abs{v}$.

Now \eqref{eq:helm1} implies
\begin{equation}
\label{eq:phase}
|u(r\theta)|^2=\frac{2\pi}{r^{d-2}}\sum_{m,n=0}^\infty 
c_{m,n}(u)J_{m+\frac{d-2}{2}}(r)J_{n+\frac{d-2}{2}}(r)
\end{equation}
where 
$$
c_{m,n}(u)=\sum_{j=1}^{N(m)}\sum_{k=1}^{N(n)}
a_{m,j}(u)\overline{a_{n,k}(u)} Y_m^j(\theta)\overline{Y_n^k(\theta)}.
$$
Note that $c_{m,n}(u)=\overline{c_{n,m}(u)}$. Therefore, Lemma \ref{lem:linind} implies that
$|u|=|v|$ is equivalent to $\Re\bigl(c_{m,n-m}(u)\bigr)=\Re\bigl(c_{m,n-m}(v)\bigr)$
for all $0\leq m\leq n$. Finally, replacing $n$ by $n+m$ and using the symmetry $c_{n,m}(u)=\overline{c_{m,n}(u)}$,
we obtain the following:

\begin{lemma} 
\label{lem:gensol}
Let $d\geq 2$ and $0\in\Omega\subset\R^d$. Let $u,v$ be two solutions of
the Helmholtz equation $(H)$ on $\Omega$.
Then $|u|=|u|$ if and only if
for every $0\leq m\leq n$
\begin{equation}
\label{eq:gensol}
\Re\bigl(c_{m,n}(u)\bigr)=\Re\bigl(c_{m,n}(v)\bigr).
\end{equation}
\end{lemma}

Note that 
$c_{m,m}(u)=\dst\abs{\sum_{k=1}^{N(m)}a_{m,k}(u)Y_m^k(\theta)}^2$.
In particular, 
\begin{equation}
\label{eq:phaseretrieval}
\Re\bigl(c_{m,m}(v)\bigr)=\Re\bigl(c_{m,m}(v)\bigr)
\end{equation}
is the classical phase retrieval problem for trigonometric polynomials. 
We are thus facing a family of classical phase retrieval problems
\eqref{eq:phaseretrieval} with compatibility relations
$\Re\bigl(c_{m,n}(v)\bigr)=\Re\bigl(c_{m,n}(v)\bigr)$, $n\not=m$.

The main difference between the 2 dimensional problem
and the higher dimensional one is that the 2 dimensional one has only trivial solutions (for fixed $m$).
We will show that the compatibility relations imply that the same trivial solution has to be chosen independently of $m$.

Further, if $\{Y_m^k\}_{k=1,\ldots,N(m)}$ is an orthonormal basis of $\hh_m^d$ then, 
integrating over $\S^{d-1}$ and using orthogonality, we obtain
\begin{equation}
	\label{eq:support}
	\sum_{k=1}^{N(m)}\abs{a_{m,k}(u)}^2=\sum_{k=1}^{N(m)}\abs{a_{m,k}(v)}^2.
\end{equation}
As a consequence if $u$ is a trigonometric polynomial
in the sense that $\{a_{m,j}(u)\}$ has finite support, so has $\{a_{m,j}(v)\}$ and $v$ is also a trigonometric polynomial.
We have thus proved the following lemma:

\begin{lemma}
	\label{lem:support}
	Let $d\geq 2$ and $0\in\Omega\subset\R^d$. Let $u,v$ be two solutions of
the Helmholtz equation $(H)$ on $\Omega$.

Assume that $u$ is a trigonometric polynomial ($K$-finite) in the sense that its expansion in spherical harmonics \eqref{eq:helm1} has only finitely many terms. If $\abs{u}=\abs{v}$, then $v$ is also a trigonometric polynomial.
\end{lemma}
%

\section{The case of dimension $d\geq3$}
\label{sec:d3}

\subsection{The real case}

Let us now start by showing that the problem is very simple if one restricts it to real valued solutions.
In this case, $|u|=|v|$ is equivalent to $u^2=v^2$ thus $(u-v)(u+v)=0$.
Thus, one of $u=v$ or $u=-v$ occurs on a set of positive measure and,
as $u,v$ are analytic, either $u=v$ or $u=-v$.

According to \cite{Le,Sj}, (see also \cite{JK}) in dimension $d=2$, and to \cite{GJ} for general dimension, solutions of the Helmholtz equation in 
$\R^d$ are uniquely determined by their restriction to two generic hyperplanes. It follows from \cite{FBGJ} that this result also holds for solutions on domains. Those results extend to the phase retrieval problem as follows:

\begin{proposition}
\label{prop:real}
Let $d\geq 2$ and $0\in\Omega\subset\R^d$. Let $u,v$ be two solutions of
the Helmholtz equation $(H)$ on $\Omega$.
Let $\theta_1,\theta_2,\theta_3\in\S^{d-1}$ and assume that $\frac{1}{\pi}\arccos\scal{\theta_j,\theta_k}\notin\Q$ when 
$j\not=k\in\{1,2,3\}$.
Assume that

--- $u,v$ are real valued;

--- $|u|=|v|$ on the hyperplanes $\Omega\cap\theta_j^\perp$, $j=1,2,3$ (a fortiori if $|u|=|v|$ on $\Omega$).

Then either $u=v$ or $u=-v$.
\end{proposition}

\begin{proof} 
As we have already noticed, we have $(u-v)(u+v)=0$ on $\Omega\cap\theta_j^\perp$.
Therefore, for each $j=1,2,3$ at least one of the following two cases 
holds

--- either $u=v$ on a subset of $\Omega\cap\theta_j^\perp$ of positive $d-1$-dimensional Lebesgue measure on the hyperplane, thus by analicity, on
$\Omega\cap\theta_j^\perp$;

--- or $u=-v$ on $\Omega\cap\theta_j^\perp$.

We thus have 2 cases to consider

--- either $u=v$ on at least 2 of $\theta_j^\perp$, $j=1,2,3$.

--- or $u=-v$ on at least 2 of $\theta_j^\perp$, $j=1,2,3$.

Up to replacing $v$ by $-v$ the second case reduces to the first one 
so that we can assume that $u=v$
on $\theta_1^\perp$, $\theta_2^\perp$. 
According to \cite{FBGJ}, this implies $u=v$ everywhere. 
\end{proof}

\subsection{Solutions of the Helmholtz equation with non vanishing mean in dimension $d\geq3$}

In this section, the basis of spherical harmonics will be the {\em real}
basis given by \eqref{eq:specspherharm}.

\begin{theorem}
Let $d\geq 3$ and $0\in\Omega\subset\R^d$ be a domain. Let $u,v$ be two solutions of
the Helmholtz equation $(H)$ on $\Omega$.
Assume that 
$$
\int_{\S^{d-1}}u(\eta\theta)\,\mbox{d}\sigma(\theta)\not=0
$$
for some $\eta>0$ such that $B(0,\eta)\subset\Omega$
and $J_{d/2-1}(\eta)\not=0$.

If $\abs{v}=\abs{u}$ then there exists 
$c\in\C$ with $\abs{c}=1$ such that, either $u=cv$ or $u=c\overline{v}$.
\end{theorem}

\begin{proof}
As $Y_0^1=1$ and all other spherical harmonics have vanishing spherical means, 
$$
a_{0,1}(u)=\dst\frac{\eta^{d/2-1}}{\sqrt{2\pi}\sigma(\S^{d-1})J_{d/2-1}(\eta)}\int_{\S^{d-1}}u(\eta\theta)\,\mbox{d}\sigma(\theta)\not=0
$$ is, up to a multiplicative constant, 
the mean of $u$ over $\eta\S^{d-1}$. 
But, \eqref{eq:gensol} for
$m=n=0$ reduces to $\abs{a_{0,1}(v)}^2=\abs{a_{0,1}(u)}^2$. In particular,
$v$ as non-zero mean as well. Further, up to changing $u$ and $v$ by
uni-modular multiples, we may then assume that
$a_{0,1}(u)=a_{0,1}(v)$ and that this quantity is \emph{real}.

Next, note that
$c_{0,n}(u)=\dst \sum_{k=1}^{N(n)}
a_{0,1}(u)\overline{a_{n,k}(f)}\overline{Y_n^k(\theta)}$
since $Y_0^1(\theta)=1$. Thus
$$
\Re\bigl(c_{0,n}(u)\bigr)=\dst a_{0,1}(u)\sum_{k=1}^{N(n)}
\Re\bigl(a_{n,k}(u)\bigr)
Y_n^k(\theta)
$$
As the $Y_n^k$'s are linearly independent, \eqref{eq:gensol} for
$m=0$ implies that 
$$
\Re\bigl(a_{n,k}(v)\bigr)=\Re\bigl(a_{n,k}(u)\bigr)
$$
thus $\Re(v)=\Re(u)$.

As $\abs{v}^2=\abs{u}^2$ we deduce that
$\abs{\Im(v)}^2=\abs{\Im(u)}^2$. But 
$\Im u,\Im v$ are {\em real} solutions of the Helmholtz equation.
Proposition \ref{prop:real} then implies that 
$$
\mbox{either }\Im v=\Im u
\mbox{ or }\Im v=-\Im u,
$$
that is $v=u$ or $v=\bar u$.
\end{proof}

\subsection{The sparse case and the zonal cases in dimension $d\geq 3$}

In this section, $d\geq 3$ and 
we fix an orthonormal basis of spherical harmonics 
$$
\{Y_m^j\}_{m\geq 0,j=1,\ldots,N(m)}
$$
that is
either zonal or is of the form \eqref{eq:specspherharm}.
In particular, it is {\em real} and has the property that if
$a(Y_m^j)^2=b(Y_n^k)^2$ on $\S^{d-1}$, then $m=n$, $j=k$ and $b=a$.
For the zonal basis this is due to Lemma \ref{lem:zonlinind} and is a consequence of Lemma \ref{lem:specbaslinind} for the other basis.

\begin{definition}
We will say that $u\in L^2(\S^{d-1})$ is {\em sparse} (in the basis $\{Y_j^m\}$) if, for every $m$ there exists at most one $j=j(m)$ such that $a_{m,j}(u)\not=0$.
\end{definition}

\begin{example}
A zonal function is sparse in a zonal basis.
\end{example}

\begin{proposition}
Let $d\geq 3$ and $0\in\Omega\subset\R^d$ be a domain. Let $u,v$ be two solutions of
the Helmholtz equation $(H)$ on $\Omega$.

Assume that both $u$ and $v$ are sparse
	in a common \emph{real} orthonormal basis of spherical harmonics $\{Y_m^j,m\geq 0, 0\leq j\leq N(m)\}$. If $\abs{v}=\abs{u}$
	then there exists a $c\in\C$ with $\abs{c}=1$ such that either $v=cu$
	or $v=c\overline{u}$.
\end{proposition}

\begin{proof}
Let $u,v$ be two sparse functions:
$$
u(r\theta)=\frac{\sqrt{2\pi}}{r^{d/2-1}}\sum_{m\geq0}a_{m,j(m)}(u)J_{\nu(m)}(r)Y_m^{j(m)}(\theta)
$$
and
$$
v(r\theta)=\frac{\sqrt{2\pi}}{r^{d/2-1}}\sum_{m\geq0}a_{m,k(m)}(v)J_{\nu(m)}(r)Y_m^{k(m)}(\theta)
$$
and assume that $\abs{u}=\abs{v}$. Note that
some of the $a_{m,j(m)}(u)$'s may still be zero. For simplicity, we set $j(m)=1$
when $a_{m,j}(u)=0$ for all $j$ and $k(m)=1$ when $a_{m,k}(v)=0$ for all $k$.

First, \eqref{eq:gensol} for $n=m$ implies that
$$
\abs{a_{m,j(m)}(u)}^2\bigl(Y_m^{j(m)}\bigr)^2
=\abs{a_{m,k(m)}(v)}^2\bigl(Y_m^{k(m)}\bigr)^2.
$$
This implies that $k(m)=j(m)$, {\it i.e} that
$a_{m,k}(u)$ and $a_{m,k}(v)$ have same support,
and that $\abs{a_{m,j(m)}(u)}=\abs{a_{m,j(m)}(v)}$.

Let $\mm=\{m\geq0\,:\ a_{m,j(m)}(u)\not=0\}$ and $m_0=\min\mm$. 
Then, up to replacing $u,v$ by unimodular multiples,
with may assume that $a_{m_0,j(m_0)}(u)=a_{m_0,j(m_0)}(v)$ is \emph{real}
(and non zero).

Next, note that, for $m\in\mm$,
$$
c_{m,n}(u)=
a_{m,j(m)}(u)\overline{a_{n,j(n)}(u)} Y_m^{j(m)}(\theta)
Y_n^{j(n)}(\theta)
$$
and similarily
$$
c_{m,n}(v)=
a_{m,j(m)}(v)\overline{a_{n,j(n)}(v)} Y_m^{j(m)}(\theta)
Y_n^{j(n)}(\theta)
$$
Further, $Y_m^{j(m)}(\theta)Y_n^{j(n)}(\theta)$ is a non-zero real polynomial. Thus \eqref{eq:gensol} reduces to
$$
\Re\bigl(a_{m,j(m)}(u)\overline{a_{n-m,j(n-m)}(u)}\bigr)
=\Re\bigl(a_{m,j(m)}(v)\overline{a_{n-m,j(n-m)}(v)}\bigr)
$$
for every $0\leq m\leq n$. Taking $m=m_0$ and $n=2m_0+k$ we get
$$
\Re\bigl(a_{m_0+k,j(m_0+k)}(u)\bigr)=
\Re\bigl(a_{m_0+k,j(m_0+k)}(v)\bigr).
$$

It follows that $\Re(u)=\Re(v)$ and, as in the previous proof, this implies that either $u=v$
or $u=\bar v$.
\end{proof}

\section{The 2 dimensional case}
\label{sec:d2}

In this section, we will specifically treat the 2 dimensional case for which uniqueness is guarantied:

\begin{theorem}
Let $0\in\Omega\subset\R^2$ be a domain. Let $u,v$ be two solutions of
the Helmholtz equation $(H)$ on $\Omega$ be such that $|v|=|u|$. Then there exists $c\in\C$ with $|c|=1$ such that
either $v=cu$ on $\Omega$ or $v=c\bar u$ on $\Omega$.
\end{theorem}

\begin {proof}
A basis of spherical harmonics is given by the usual Fourier basis
$\{e^{ik\theta},k\in\Z\}$ so that we  may write $u$ in polar coordinates as a Fourier series
$$
u(r\theta)=\sqrt{2\pi}\sum_{k\in\Z}\widehat{u}(k)J_{\abs{k}}(r)e^{ik\theta}.
$$
Note that $c_{m,n}(u)=$ 
$$
\begin{cases}
|\hat u(0)|^2&\mbox{if }m=n=0\\
\hat u(0)\overline{\bigl(\hat u(-n)e^{-in\theta}+\hat u(n)e^{in\theta}\bigr)}&\mbox{if }m=0,n>0\\
\bigl(\hat u(-m)e^{-im\theta}+\hat u(m)e^{im\theta}\bigr)\overline{\bigl(\hat u(-n)e^{-in\theta}+\hat u(n)e^{in\theta}\bigr)}
&\mbox{if }m,n>0\end{cases}.
$$
We will now exploit Lemma \ref{lem:gensol}, that is $\Re c_{m,n}(u)=\Re c_{m,n}(v)$ for every $0\leq m\leq n$.

We have already treated the case $\widehat{u}(0)\not=0$ so that we will now assume that $\widehat{u}(0)=0$.
This implies that $c_{0,0}(v)=c_{0,0}(u)=0$, thus $\widehat{v}(0)=0$.

Next $c_{m,m}(u)=\abs{\hat u(-m)e^{-im\theta}+\hat u(m)e^{im\theta}}^2$ is real so that
\begin{equation}
\label{eq:phase2d}
\abs{\hat v(-m)e^{-im\theta}+\hat v(m)e^{im\theta}}^2=\abs{\hat u(-m)e^{-im\theta}+\hat u(m)e^{im\theta}}^2
\end{equation}
for every $m\geq 1$. We thus need the following lemma:

\begin{lemma}\label{lem:phase2d}
Let $m,n\in\N$. Let $a,b,c,d\in\C$ be such that, for every $\theta\in\R$,
\begin{equation}
\label{eq:phase2dlem}
\abs{ae^{im\theta}+be^{-im\theta}}^2=\abs{ce^{im\theta}+de^{-im\theta}}^2
\end{equation}
then there exists $\kappa\in\C$ such that either $ae^{im\theta}+be^{-im\theta}=\kappa\bigl(ce^{im\theta}+de^{-im\theta}\bigr)$
for every $\theta$ or $ae^{im\theta}+be^{-im\theta}=\kappa\overline{\bigl(ce^{im\theta}+de^{-im\theta}\bigr)}$
for every $\theta$.
\end{lemma}

\begin{proof}[Proof of Lemma \ref{lem:phase2d}] This lemma is folklore in the subject, for sake of completeness, let us here give the proof.
Expanding the square in \eqref{eq:phase2dlem} we see that
$|a|^2+|b|^2=|c|^2+|d|^2$ and $a\bar b=c\bar d$. In particular $|a||b|=|c||d|$, so that $|a|^2,|b|^2$ have same sum and product as
$|c|^2,|d|^2$. There are thus two possibilities, either $|a|=|c|$ and $|b|=|d|$ or $|a|=|d|$ and $|b|=|c|$.
We may thus write $a=ce^{i\ffi}$ and $b=de^{i\psi}$ (resp. $a=\bar de^{i\ffi}$ and $b=\bar ce^{i\psi}$)
so that $c\bar de^{i(\ffi-\psi)}=c\bar d$ (resp. $c\bar de^{i(\ffi-\psi)}=c\bar d$).
We now distinguish 3 cases:
\begin{enumerate}
\renewcommand{\theenumi}{\roman{enumi}}
\item If $c\bar d\not=0$ then $e^{i\ffi}=e^{i\psi}$ and the conclusion is straightforward.

\item If $c=0$ then $a=0$ --- resp. $b=0$--- and 
$$
ae^{im\theta}+be^{-im\theta}=e^{i\psi}de^{-im\theta}=e^{i\psi}(ce^{im\theta}+de^{-im\theta})
$$
--- resp. $ae^{im\theta}+be^{-im\theta}=e^{i\ffi}\bar de^{im\theta}=e^{i\ffi}\overline{(ce^{im\theta}+de^{-im\theta})}$.

\item If $d=0$ then $b=0$ --- resp. $a=0$--- and $ae^{im\theta}+be^{-im\theta}=e^{i\ffi}(ce^{im\theta}+de^{-im\theta})$
--- resp. $ae^{im\theta}+be^{-im\theta}=e^{i\psi}\overline{(ce^{im\theta}+de^{-im\theta})}$.
\end{enumerate}
The proof of the lemma is thus complete.
\end{proof}

Applying this lemma, we may now distinguish two cases:

\smallskip

\noindent{\bf Type I.} We will say that $m$ is of {\em type I} if $\bigl(\hat{u}(m),\hat{u}(-m)\bigr)\not=(0,0)$ and if
there exists $\kappa_m\in\C$ with $|\kappa_m|=1$ such that
$$
\hat v(-m)e^{-im\theta}+\hat v(m)e^{im\theta}=\kappa_m\bigl(\hat u(-m)e^{-im\theta}+\hat u(m)e^{im\theta}\bigr).
$$

\smallskip

\noindent{\bf Type C.} We will say that $m$ is of {\em type C} if $\bigl(\hat{u}(m),\hat{u}(-m)\bigr)\not=(0,0)$ and if
there exists $\kappa_m\in\C$ with $|\kappa_m|=1$ such that
$$
\hat v(-m)e^{-im\theta}+\hat v(m)e^{im\theta}=\kappa_m\overline{\bigl(\hat u(-m)}e^{-im\theta}+\hat u(m)e^{im\theta}\bigr).
$$

\smallskip

It should be noted that $m$ can be of both types simultaneously. This happens precisely when $\abs{\hat u(-m)}=\abs{\hat u(m)}$,
in this case, we may write $\hat u(-m)=e^{-2i\theta_m}\hat u(m)$ so that 
$$
\hat u(-m)e^{-im\theta}+\hat u(m)e^{im\theta}=e^{-i\theta_m}\hat u(m)\cos (m\theta+\theta_m).
$$

\smallskip

\noindent{\bf Type R.} We will say that $m$ is of {\em type R} if $\abs{\hat{u}(m)}=\abs{\hat{u}(-m)}\not=0$ and if
there exists $\kappa_m\in\C$ with $|\kappa_m|=1$, $\theta_m\in\R$ such that
$$
\left\{\begin{matrix}
\hat u(-m)e^{-im\theta}+\hat u(m)e^{im\theta}&=&e^{-i\theta_m}\hat u(m)\cos (m\theta+\theta_m)\\
\hat v(-m)e^{-im\theta}+\hat v(m)e^{im\theta}&=&\kappa_me^{-i\theta_m}\hat u(m)\cos (m\theta+\theta_m)
\end{matrix}\right.
$$

\smallskip

 Also, $\bigl(\hat{g}(m),\hat{g}(-m)\bigr)\not=(0,0)$
if $m$ has a type.

We will now need the following lemma:

\begin{lemma}
\label{lem:phase2dmixt}
Let $m\not=n\in\N$.
Let $a,b,c,d\in\C$ be such that $(a,b),(c,d)\not=(0,0)$. 
Let $\kappa,\kappa'\in\C$ with $|\kappa|=|\kappa'|=1$.
Assume that
\begin{multline*}
\Re\bigl(\kappa'\bar\kappa(ae^{im\theta}+be^{-im\theta})(ce^{in\theta}+de^{-in\theta})\bigr)\\
=\Re\bigl((ae^{im\theta}+be^{-im\theta})(ce^{in\theta}+de^{-in\theta})\bigr)
\end{multline*}
for every $\theta\in\R$. 
\begin{enumerate}
\renewcommand{\theenumi}{\roman{enumi}}
\item If $|a|\not=|b|$ or $|c|\not=|d|$, then $\kappa'=\kappa$.

\item If $|a|=|b|$ and $|c|=|d|$, write $b=e^{-2i\ffi}a$, $d=e^{-2i\psi}c$ and
$\alpha=e^{-i\ffi}a$, $\beta=e^{-i\psi}c$.
Then either $\kappa'=\kappa$ or $\dst\kappa'= \frac{\overline{\alpha\beta}}{\alpha\beta}\kappa$.
\end{enumerate}
\end{lemma}

\begin{proof}[Proof of Lemma \ref{lem:phase2dmixt}]
Let us first observe that, if $\Re(uz)=\Re(z)$ with $u,z\in\C$, $|u|=1$ and $z\not=0$ then either $uz=z$ and $u=1$ or $uz=\bar z$.

We thus have to prove that the second case only occurs when $|a|=|b|$ and $|c|=|d|$ and that $\kappa,\kappa'$
are then related by $\dst\kappa'= \frac{\overline{\alpha\beta}}{\alpha\beta}\kappa$.
\begin{multline}
\label{eq:conj}
\kappa'\bar\kappa(ae^{im\theta}+be^{-im\theta})(ce^{in\theta}+de^{-in\theta})\\
=\overline{(ae^{im\theta}+be^{-im\theta})(ce^{in\theta}+de^{-in\theta})}
\end{multline}
for a set of positive measure of $\theta$'s thus, by analycity, for all $\theta$.

Expanding and comparing the coefficients of $e^{\pm im\theta\pm in\theta}$ (recall that $m\not=n$) this is equivalent to
$$
(i)\ \kappa'\bar\kappa ac=\overline{bd},\quad (ii)\ \kappa'\bar\kappa ad=\overline{bc}
,\quad (iii)\ \kappa'\bar\kappa bc=\overline{ad}\quad\mbox{and}\quad (iv)\ \kappa'\bar\kappa bd=\overline{ac}.
$$

First, if $a=0$ (resp. $b=0$) then, as $b\not=0$ (resp $a\not=0$), the two first equations imply $d=c=0$, a contradiction.
Thus $a\not=0$ and $b\not=0$ and any of the equations then shows that $c\not=0$ and $d\not=0$ since otherwise they would both be $0$.

Next,  $(i)/(ii)$ reads $\dst\frac{c}{d}=\frac{\bar d}{\bar c}$ thus $|c|^2=|d|^2$ and comparing modulus in $(i)$ then shows that
$|a|=|b|$. We then write $b=e^{-2i\ffi}a$, $d=e^{-2i\psi}c$. But then, \eqref{eq:conj} reads
\begin{multline*}
\kappa'\bar\kappa e^{-i(\ffi+\psi)} ac(e^{i(m\theta+\ffi)}+e^{-i(m\theta+\ffi)})(e^{i(n\theta+\psi)}+e^{-i(n\theta+\psi)})\\
=\overline{ace^{-i(\ffi+\psi)}(e^{i(m\theta+\ffi)}+e^{-i(m\theta+\ffi)})(e^{i(n\theta+\psi)}+e^{-i(n\theta+\psi)})}
\end{multline*}
which reduces to $\dst\kappa'= \frac{\overline{\alpha\beta}}{\alpha\beta}\kappa$.
\end{proof}

As a consequence of the lemma we get that, if $m,n$ are of type I (resp. of type C) and one of them is not of type $R$, then 
$\kappa_m=\kappa_n$. We thus obtain the following:

--- either all $m$ are of type R and then
\begin{equation}
\label{eq:pure}
\left\{\begin{matrix}
u(r\theta)&=&\dst\sqrt{2\pi}\sum_{m=1}^\infty \beta_mJ_m(r)\cos (m\theta+\theta_m)\\[15pt]
v(r\theta)&=&\dst\sqrt{2\pi}\sum_{m=1}^\infty \kappa_m\beta_mJ_m(r)\cos (m\theta+\theta_m)
\end{matrix}
\right.
\end{equation}
where $\beta_m=\widehat{u}(m)e^{-i\theta_m}$ and for every $m,n$
for which $\widehat{u}(m)\widehat{u}(n)\not=0$, 
$$
\mbox{either}\quad\kappa_m=\kappa_n\quad\mbox{or}\quad\kappa_m=\frac{\overline{\beta_m}}{\beta_m}\frac{\beta_n}{\overline{\beta_n}}\kappa_n.
$$

--- or there exists exactly at least one $m$ that is not of type $R$ and at least one $m$ of type $R$ (thus of type R and C)
and then $\kappa_m=\kappa_n:=\kappa$ for every $m,n$ and
\begin{multline}
\label{eq:mixt1}
v(r\theta)=\kappa\sqrt{2\pi} \left(\sum_{m\mbox{ of type I not R}}J_m(r)\bigl(\hat u(-m)e^{-im\theta}+\hat u(m)e^{im\theta}\bigr)\right.\\
+\sum_{m\mbox{ of type R}}\hat u(-m)e^{-i\theta_m}J_m(r)\cos(m\theta+\theta_m)\\
+\left.\sum_{m\mbox{ of type C not R}}
J_m(r)\overline{\bigl(\hat u(-m)e^{-im\theta}+\hat u(m)e^{im\theta}\bigr)}\right);
\end{multline}

--- or no $m$ is of type R and then there exists $\kappa,\tilde\kappa$ such that
\begin{multline}
\label{eq:mixt2}
v(r\theta)=\kappa\sqrt{2\pi} \sum_{m\mbox{ of type }I}J_m(r)\bigl(\hat u(-m)e^{-im\theta}+\hat u(m)e^{im\theta}\bigr)\\
+\tilde\kappa\sum_{m\mbox{ of type }C}J_m(r)\overline{\bigl(\hat u(-m)e^{-im\theta}+\hat u(m)e^{im\theta}\bigr)}.
\end{multline}

We will now show that in the two last cases one of the type I or type C sums in \eqref{eq:mixt1}-\eqref{eq:mixt2} is empty.
This follows from the following lemma and the fact that $\Re c_{m,n}(v)=\Re c_{m,n}(u)$ with $m$
of type I but not R and $n$ of type C but not R:

\begin{lemma}
\label{lem:mixt}
Let $m\not=n\in\N$.
Let $a,b,c,d\in\C$ be such that $(a,b)\not=(0,0)$ and $(c,d)\not=(0,0)$.
Let $\kappa,\kappa'\in\C$ with $|\kappa|=|\kappa'|=1$.
If
\begin{multline*}
\Re\bigl(\kappa'\bar\kappa(ae^{im\theta}+be^{-im\theta})\overline{(ce^{in\theta}+de^{-in\theta})}\bigr)\\
=\Re\bigl((ae^{im\theta}+be^{-im\theta})(ce^{in\theta}+de^{-in\theta})\bigr)
\end{multline*}
for every $\theta\in\R$ then $|c|=|d|$.
\end{lemma}

We postpone the proof of the lemma to the end of the proof of the theorem.
Let us first conclude with the first case \eqref{eq:pure}. Let $m_0=\inf\{m\,:\beta_m\not=0\}$ and let $\tilde\kappa_{m_0}=\frac{\beta_{m_0}}{\overline{\beta_{m_0}}}\kappa_{m_0}$. We may then write $\mm_0:=\{m\,:\beta_m\not=0,\kappa_m=\kappa_{m_0}\}$,
$\mm_1:=\{m\,:\beta_m\not=0,\kappa_m\not=\kappa_{m_0}\}$.
If $\mm_1=\emptyset$ then $g=\kappa_{m_0}f$ otherwise let $m_1=\min\mm_1$.
Then for every $m\in\mm_1$, $\kappa_m\beta_m=\tilde\kappa_{m_0}\overline{\beta_m}$.
Further, if $n\in\mm_0$ and $m\in\mm_1$ then $\kappa_n\not=\kappa_m$ thus
$$
\kappa_n\beta_n=\overline{\beta_n}\frac{\beta_m\kappa_m}{\overline{\beta_m}}=\tilde\kappa_{m_0}\overline{\beta_n}.
$$
But then
\begin{eqnarray*}
g(\theta)&=&\kappa_{m_0}\sum_{m\in\mm_0}\beta_m\cos(m\theta+\theta_m)
+\tilde\kappa_{m_0}\sum_{m\in\mm_1}\overline{\beta_m}\cos(m\theta+\theta_m)\\
&=&\tilde\kappa_{m_0}\sum_{m\geq 1}\overline{\beta_m}\cos(m\theta+\theta_m)=\tilde\kappa_{m_0}\overline{f}.
\end{eqnarray*}

\end{proof}

\begin{proof}[Proof of Lemma \ref{lem:mixt}] We will use the fact that two complex numbers of same modulus
and same real part are either equal of conjugate of one an other. Therefore

--- either 
\begin{multline}
\label{eq:2dcas1}
\kappa'\bar\kappa(ae^{im\theta}+be^{-im\theta})\overline{(ce^{in\theta}+de^{-in\theta})}\\
=(ae^{im\theta}+be^{-im\theta})(ce^{in\theta}+de^{-in\theta})
\end{multline}

-- or
\begin{multline}
\label{eq:2dcas2}
\kappa'\bar\kappa(ae^{im\theta}+be^{-im\theta})\overline{(ce^{in\theta}+de^{-in\theta})}\\
=\overline{(ae^{im\theta}+be^{-im\theta})(ce^{in\theta}+de^{-in\theta})}
\end{multline}

Note that, by analycity, if one of the alternatives holds for a set of positive measure of $\theta$'s, then it holds
for every $\theta$. Further, up to exchanging the roles of the two factors, \eqref{eq:2dcas2} reduces to \eqref{eq:2dcas1}.
But, expanding the factors, we see that this equation is equivalent to
$$
\begin{matrix}
(i)& \kappa'\bar\kappa a\bar d&=&ac&(ii)&\kappa'\bar\kappa a\bar c&=&ad\\
(iii)&\kappa'\bar\kappa b\bar c&=&bd&(iv)&\kappa'\bar\kappa b\bar d&=&bc.
\end{matrix}
$$
Now, if $c=0$ then $d\not=0$ and (i) implies that $a=0$ while (iv) implies $b=0$ which is excluded.
Using (ii) and (iii) we can also exclude $d=0$. 

Further (i),(ii) imply that $|a||c|=|a||d|$ while (iii),(iv) imply $|b||c|=|b||d|$.
As one of $a,b\not=0$ we get $|c|=|d|$ which was excluded.
\end{proof}

\section*{Acknowledgments}

This study has been carried out with financial support from the French State, managed
by the French National Research Agency (ANR) in the frame of the ”Investments for
the future” Programme IdEx Bordeaux - CPU (ANR-10-IDEX-03-02).

P.J. acknowledges financial support from the French ANR program ANR-12-BS01-0001 (Aventures)
from the Austrian-French AMADEUS project 35598VB - ChargeDisq and from the Tunisian-French CMCU/Utique project 15G1504.

S.P.E. acknowledges financial support from the Mexican Grant PAPIIT-UNAM IN102915.

\end{document}